\documentclass[12pt]{amsart}
\usepackage{paralist}
\usepackage{amssymb}
\usepackage{amstext}
\usepackage{amsmath}
\usepackage{amscd}
\usepackage{latexsym}
\usepackage{amsfonts}
\usepackage{color}
\theoremstyle{plain}
\newtheorem{thm}{Theorem}[section]
\newtheorem*{thm*}{Theorem}
\newtheorem*{cor*}{Corollary}

\newtheorem{prop}[thm]{Proposition}
\newtheorem{lem}[thm]{Lemma}
\newtheorem{cor}[thm]{Corollary}
\newtheorem{claim}{Claim}
\newtheorem*{claim*}{Claim}

\theoremstyle{definition}
\newtheorem{defn}[thm]{Definition}
\newtheorem{ex}[thm]{Example}
\newtheorem{rem}[thm]{Remark}

\theoremstyle{remark}

\numberwithin{equation}{thm}
\def\Ker{\mathrm{Ker}}
\def\Im{\mathrm{Im}}

\def\rank{\mathrm{rank}}
\def\a{\mathfrak a}

\def\e{\mathrm{e}}
\def\m{\mathfrak m}
\def\n{\mathfrak n}

\def\q{\mathfrak q}

\def\Z{\Bbb Z}

\newcommand{\rmG}{\mathrm{G}}

\newcommand{\rmR}{\mathrm{R}}
\newcommand{\rmS}{\mathrm{S}}

\def\depth{\operatorname{depth}}

\def\Ann{\mathrm{Ann}}
\def\Ass{\mathrm{Ass}}

\newcommand{\mapdown}[1]{\Big\downarrow
\llap{$\vcenter{{$\scriptstyle#1\,$}}$ }}

\begin{document}

\setlength{\baselineskip}{12pt}
%%%%%%%%%%%%%%%%%%%%%%%%%%%%%%%%%%%%%%%%%%%%%%%%%%%%%%%%%%%%%
%%%%%%%%%%%%%%%%%%%%%%%%%%%%%%%%%%%%%%%%%%%%%%%%%%
\title{The reduction number of stretched ideals}
\pagestyle{plain}
\author{Kazuho Ozeki}
\address{Department of Mathematical Sciences, Faculty of Science, Yamaguchi University, 1677-1 Yoshida, Yamaguchi 753-8512, Japan}
\email{ozeki@yamaguchi-u.ac.jp}
%\dedicatory{}
\thanks{
\endgraf
\endgraf
The author was partially supported by Grant-in-Aid for Scientific Researches (C) in Japan (18K03241).}
\maketitle
%%%%%%%%%%%%%%%%%%%%%%%%%%%%%%%%%%%%%%%%%%%%%%%%%%%%%%%%%%%%%
%%%%%%%%%%%%%%%%%%%%%%%%%%%%%%%%%%%%%%%%%%%%%%%%%%%%%%%%%%%%%
\begin{abstract}
The homological property of the associated graded ring of an ideal is an important problem in commutative algebra and algebraic geometry.
In this paper we explore the almost Cohen-Macaulayness of the associated graded ring of stretched $\m$-primary ideals in the case where the reduction number attains almost minimal value in a Cohen-Macaulay local ring $(A,\m)$.
As an application, we present complete descriptions of the associated graded ring of stretched $\m$-primary ideals with small reduction number.
\end{abstract}

%{\footnotesize
%\tableofcontents
%}
%%%%%%%%%%%%%%%%%%%%%%%%%%%%%%%%%%%%%%%%%%%%%%%%%%%%%%%%%%%%%
%%%%%%%%%%%%%%%%%%%%%%%%%%%%%%%%%%%%%%%%%%%%%%%%%%%%%%%%%%%%%
%%%%%%%%%%%%%%%%%%%%%%%%%%%%%%%%%%%%%%%%%%%%%%%%%%%%%%%%%%%%%
%%%%%%%%%%%%%%%%%%%%%%%%%%%%%%%%%%%%%%%%%%%%%%%%%%%%%%%%%%%%%
%%%%%%%%%%%%%%%%%%%%%%%%%%%%%%%%%%%%%%%%%%%%%%%%%%%%%%%%%%%%%
%%%%%%%%%%%%%%%%%%%%%%%%%%%%%%%%%%%%%%%%%%%%%%%%%%%%%%%%%%%%%

\section{Introduction}

Throughout this paper, let $A$ be a Cohen-Macaulay local ring with maximal ideal $\m$ and $d=\dim A>0.$ 
For simplicity, we may assume the residue class field $A/\m$ is infinite.
Let $I$ be an $\m$-primary ideal in $A$ and let
$$ R=\rmR(I) := A[I t] \ \ \ \subseteq A[t] ~~\operatorname{and}~~ \ R'= \rmR'(I):= A[It,t^{-1}] \ \ \subseteq A[t,t^{-1}]$$
denote, respectively, the Rees algebra and the extended Rees algebra of $I$. 
Let
$$ G=\rmG(I):= R'/t^{-1}R' \cong \bigoplus_{n \geq 0}I^n/I^{n+1}$$
denote the associated graded ring of $I$.
Let $\ell_A(N)$ denote, for an $A$-module $N$, the length of $N$.

Let $Q=(a_1,a_2,\cdots,a_d) \subseteq I$ be a parameter ideal in $A$ which forms a reduction of $I$.
We set 
$$n_{I}=n_{Q}(I):=\min\{n \geq 0 \ | \ I^{n+1} \subseteq Q \} \ \ \ \mbox{and} \ \ \ r_{I}=r_Q(I):=\min\{n \geq 0 \ | \ I^{n+1}=QI^n\},$$
respectively, denote the index of nilpotency and the reduction number of $I$ with respect to $Q$. 
Then it is easy to see that the inequality ${r_I} \geq {n_I}$ always holds true.

\vspace{2mm}

In 1967 Abhyankar gave the inequality $\e_0(\m) \geq \mu(\m)-d+1$, where $\e_0(\m)$ and $\mu(\m)$ denote the multiplicity and the embedding dimension of $A$ respectively.
Sally investigated Cohen-Macaulay local rings with $\e_0(\m)=\mu(\m)-d+1$ $($resp. $\e_0(\m)=\mu(\m)-d+2)$ which have been called {\it rings of minimal multiplicity} $($resp. {\it rings of almost minimal multiplicity}$)$.
In the case of minimal multiplicity, she showed that the associated graded ring ${\rm G}(\m)$ of the maximal ideal $\m$ is always Cohen-Macaulay $([18])$.
In [20], she proved that the associated graded ring ${\rm G}(\m)$ is Cohen-Macaulay in the case where the rings having almost minimal multiplicity and the type of the ring doesn't attend to maximal value, while she found examples of Cohen-Macaulay local rings whose associated graded ring ${\rm G}(\m)$ is not Cohen-Macaulay.
Based on these considerations, Sally expected that the associated graded ring ${\rm G}(\m)$ is almost Cohen-Macaulay, i.e. the depth is at least $d-1$, if the base local ring $A$ has almost minimal multiplicity.
This {\it Sally's conjecture} was solved by Rossi-Valla [15], and independently Wang [23].
Thereafter, in [16], the almost Cohen-Macaulay property of the associated graded rings with $\e_0(\m)=\mu(\m)-d+3$ was examined by Rossi and Valla.

The notion of {\it stretched} Cohen-Macaulay local rings was introduced by J. Sally to extend the rings of minimal or almost minimal multiplicity.
We say that the ring $A$ is {\it stretched} if $\ell_A(\m^2 + Q/\m^3 +Q)=1$ holds true, i.e. the ideal $(\m/Q)^2$ is principal, for some parameter ideal $Q$ in $A$ which forms a reduction of $\m$ $([16])$.
We note here that this condition depends on the choice of a reduction $Q$ $($see [17, Example 2.3]$)$.
She showed that the equality ${{r}}_Q(\m)={{n}}_Q(\m)$ holds true if and only if the associated graded ring ${\rmG}(\m)$ of $\m$ is Cohen-Macaulay in the case where the base local ring $A$ is stretched.

In 2001, Rossi and Valla [17] gave the notion of stretched $\m$-primary ideals.
We say that the $\m$-primary ideal $I$ is stretched if the following two conditions
\begin{itemize}
\item[$(1)$] $Q \cap I^2=QI$ and
\item[$(2)$] $\ell_A(I^2 + Q/I^3 + Q)=1$
\end{itemize}
hold true for some parameter ideal $Q$ in $A$ which forms a reduction of $I$.
We notice that the first condition is naturally satisfied if $I=\m$ so that this extends the classical definition of stretched local rings given in [19].

Throughout this paper, a stretched $\m$-primary ideal $I$ will come always equipped with a parameter ideal $Q$ in $A$ which forms a reduction of $I$ such that $Q \cap I^2=QI$ and $\ell_A(I^2+Q/I^3+Q)=1$.
Rossi and Valla [17] showed that the equality ${{r}_I}={{n_I}}$ holds true if and only if the associated graded ring $G$ is Cohen-Macaulay in the case where $I$ is stretched.
Thus stretched $\m$-primary ideals whose reduction number attends to minimal value enjoy nice properties.
They also showed that the associated graded ring $G$ is almost Cohen-Macaulay if $I$ is stretched and $I^{n_I+1} \subseteq QI^{n_I-1}$ holds true $($[17, Proposition 3.1]$)$. 
Recently, Mantelo and Xie [11] introduced the notion of $j$-stretched ideals and generalized the results of Sally and Rossi-Valla.
An interesting result on the depth of fiber cones of stretched $\m$-primary ideals is presented in [10].

The Hilbert coefficients of $\m$-primary ideal in a local ring are important numerical invariants associated to an ideal.
In this paper we will also study the Hilbert coefficients of stretched $\m$-primary ideals.
As is well known, for a given $\m$-primary ideal $I$, there exist integers $\{\e_k(I)\}_{0 \leq k \leq d}$ such that the equality
$$\ell_A(A/I^{n+1})={\e}_0(I)\binom{n+d}{d}-{\e_1}(I)\binom{n+d-1}{d-1}+\cdots+(-1)^d{\e}_d(I)$$
holds true for all integers $n \gg 0$.
For each $0 \leq k \leq d$, $\e_k(I)$ is called the $k$-th {\it Hilbert coefficient} of $I$. 
We set the power series $$HS_I(z)=\sum_{n =0}^{\infty}\ell_A(I^n/I^{n+1})z^n$$ and call it the Hilbert series of $I$.
It is also well known that this series is rational and that there exists a polynomial $h_I(z)$ with integer coefficients such that $h_I(1) \neq 0$ and  
$$HS_I(z)=\frac{h_I(z)}{(1-z)^d}.$$

The purpose of this paper is to explore the almost Cohen-Macaulayness of associated graded ring of stretched $\m$-primary ideal $I$ in the case where the reduction number attains almost minimal value.

The main result of this paper is the following.

\begin{thm}
Suppose that $I$ is stretched and assume that ${r}_I={n}_I+1$ is satisfied.
Then the following assertions hold true where $s=\min\{n  \geq 1 \ | \ Q \cap I^{n+1} \neq QI^n\}$.
\begin{itemize}
\item[$(1)$] $\e_1(I)=\e_0(I)-\ell_A(A/I)+\binom{n_I+1}{2}-s+1$,
\item[$(2)$] $\e_k(I) = \binom{n_I+2}{k+1}-\binom{s}{k}$ for all $2 \leq k \leq d$,
\item[$(3)$] $\ell_A(A/I^{n+1})=\sum_{k=0}^d(-1)^k\e_k(I)\binom{n+d-k}{d-k}$ for all $n \geq \max\{0, n_I-d+1\}$, and
\item[$(4)$] the Hilbert series $HS_I(z)$ of $I$ is given by {\small { $$ HS_I(z)=\frac{\ell_A(A/I)+\{\e_0(I)-\ell_A(A/I)-n_I+1\}z+\sum_{2 \leq i \leq n_I+1, i \neq s}z^i}{(1-z)^d},$$}} 
\item[$(5)$] $\depth G=d-1$.
\end{itemize}
\end{thm}

\begin{cor}
Suppose that $I$ is stretched and assume that $I^{n_I+2}=QI^{n_I+1}$ $($i.e. $r_I \leq n_I+1)$, then $G$ is almost Cohen-Macaulay $($i.e. $\depth G \geq d-1)$.
\end{cor}

As an application, we give almost Cohen-Macaulayness of the associated graded ring of stretched $\m$-primary ideals with reduction number at most four as follows.

\begin{cor}[Corollary 5.4]
Suppose that $I$ is stretched and assume that $I^5=QI^4$ $($i.e. $r_I \leq 4)$, then $G$ is almost Cohen-Macaulay $($i.e. $\depth G \geq d-1)$.
\end{cor}

Let us briefly explain how this paper is organized.
We shall prove Theorem 1.1 in Section 4.
In section 2 we will introduce some auxiliary results on the stretched $\m$-primary ideals.
In Section 3 we will introduce some techniques of the Sally module $S={\rmS}_Q(I)=I{\rmR}(I)/I{\rm R}(Q)$ for computing to the Hilbert coefficients of $I$, some of them are stated in a general setting.  
In Section 5 we shall discuss some applications of Theorem 1.1. 
We shall explore the stretched $\m$-primary ideals $I$ with small reduction number $r_{I}$.
In Section 6, examples of one-dimensional stretched Cohen-Macaulay local rings $(A, \m)$ with arbitrarily high reduction number $r_{\m}$ against the index of nilpotency $n_{\m}$ will be given $($Theorem 6.6$)$.
We will also construct a example of stretched maximal ideal which satisfy conditions of Theorem 1.1 $($Corollary 6.7$)$.

%%%%%%%%%%%%%%%%%%%%%%%%%%%%%%%%%%%%%%%%%%%%%%%%%%%
%%%%%%%%%%%%%%%%%%%%%%%%%%%%%%%%%%%%%%%%%%%%%%%%%%%
%%%%%%%%%%%%%%%%%%%%%%%%%%%%%%%%%%%%%%%%%%%%%%%%%%%
%%%%%%%%%%%%% Section 2 %%%%%%%%%%%%%%%%%%%%%%%%%%%%%
%%%%%%%%%%%%%%%%%%%%%%%%%%%%%%%%%%%%%%%%%%%%%%%%%%%
%%%%%%%%%%%%%%%%%%%%%%%%%%%%%%%%%%%%%%%%%%%%%%%%%%%

\section{Preliminary Steps}

The purpose of this section is to summarize some results on the structure of the stretched $\m$-primary ideals, which we need throughout this paper.

We set $\alpha_n=\ell_A(I^{n+1}/QI^n)$ and $\beta_n=\ell_A(QI^{n-1} \cap I^{n+1}/QI^n)$ for $n \geq 1$.
We then have the following lemma which was given by Rossi and Valla.

\begin{lem}{$([15, \mathrm{Lemma} \ 2.4])$}
Suppose that $I$ is stretched.
Then we have the following.
\begin{itemize}
\item[$(1)$] There exists $x, y \in I \backslash Q$ such that $I^{n+1}=QI^n+(x^ny)$ holds true for all $n \geq 1$.
\item[$(2)$] The map $$ I^{n+1}/QI^n \overset{\widehat{x}}{\to}I^{n+2}/QI^{n+1} $$ is surjective for all $n \geq 1$. Therefore $\alpha_n \geq \alpha_{n+1}$ for all $n \geq 1$.
\item[$(3)$] $\m x^{n}y \subseteq QI^{n}+I^{n+2}$ and hence $\ell_A(I^{n+1}/QI^n+I^{n+2}) \leq 1$ for all $n \geq 1$.
\end{itemize}
\end{lem}

The following lemma seems well known, but let us give a proof of it for the sake of completeness.

\begin{lem}
Suppose that $I$ is stretched.
Then we have the following.
\begin{itemize}
\item[$(1)$] $\alpha_1=\ell_A(I^2+Q/Q)=n_I-1$.
\item[$(2)$] $\beta_n \leq 1$ for all $n \geq 1$.
\item[$(3)$] $\alpha_n=n_I-n+\sum_{k=1}^n \beta_k$ for $1 \leq n \leq r_I-1$.
\end{itemize}
\end{lem}

We notice here that $\alpha_1=\ell_A(I^2/QI)=\e_0(I)+(d-1)\ell_A(A/I)-\ell_A(I/I^2)$ holds true (see for instance [17]), so that  $n_I=\alpha_1+1$ doesn't depend on a minimal reduction $Q$ of $I$ for stretched $\m$-primary ideals $I$.

\begin{proof}[Proof of Lemma 2.2]
$(1)$: Consider the sequence
$$ I^2+Q \supsetneq I^3+Q \supsetneq \cdots \supsetneq I^{n+1}+Q \supsetneq I^{n+2}+Q \supsetneq \cdots \supsetneq I^{n_I+1}+Q=Q.$$
Then since $I^{n+1}+Q/I^{n+2}+Q \cong I^{n+1}/Q \cap I^{n+1}+I^{n+2}$ and $$0 < \ell_A(I^{n+1}+Q/I^{n+2}+Q)=\ell_A( I^{n+1}/Q \cap I^{n+1}+I^{n+2}) \leq \ell_A(I^{n+1}/QI^n+I^{n+2}) \leq 1$$
by Lemma 2.1 $(3)$, we have $\ell_A(I^{n+1}+Q/I^{n+2}+Q)=1$ for all $1 \leq i \leq n_I-1$.
Hence, we have $$\alpha_1=\ell_A(I^2+Q/Q)= \sum_{n=1}^{n_I-1}\ell_A(I^{n+1}+Q/I^{n+2}+Q)=n_I-1$$
as required.

\noindent
$(2)$: We only need to show that the case where $2 \leq n \leq r_I-1$.
For all $2 \leq n \leq r_I-1$, we have $\ell_A(I^{n}/QI^{n-1}+I^{n+1})=1$ by Lemma 2.1 $(3)$ and $QI^{n-1}+I^{n+1}/QI^{n-1} \cong I^{n+1}/QI^{n-1} \cap I^{n+1}$ so that the equalities 
\begin{eqnarray*}
\alpha_{n-1}=\ell_A(I^{n}/QI^{n-1})&=&\ell_A(I^{n}/QI^{n-1}+I^{n+1})+\ell_A(QI^{n-1}+I^{n+1}/QI^{n-1})\\
&=& 1+\ell_A(I^{n+1}/QI^{n-1} \cap I^{n+1})\\
&=& 1+\ell_A(I^{n+1}/QI^{n})-\ell_A(QI^{n-1} \cap I^{n+1}/QI^{n})\\
&=& 1+\alpha_{n}-\beta_{n} 
\end{eqnarray*}
hold true.
Thus, we get $\beta_n = 1+\alpha_{n}-\alpha_{n-1} \leq 1$ by Lemma 2.1 $(2)$.

\noindent
$(3)$ We may assume that $n \geq 2$ and that our assertion holds true for $n-1$.
Then, we have $\alpha_n=\alpha_{n-1}+\beta_n-1$ by the proof of assertion $(2)$, and 
$$\alpha_n=\{n_I-(n-1)+\sum_{k=1}^{n-1}\beta_k\}+\beta_n-1=n_I-n+\sum_{k=1}^n \beta_k$$
by the hypothesis of induction on $n$.
\end{proof}

We set $$\Lambda:=\Lambda_I=\Lambda_Q(I)=\{n \geq 1 \ | \ QI^{n-1}\cap I^{n+1}/QI^n \neq (0) \}$$
and $|\Lambda|$ denotes the cardinality of the set $\Lambda$.
Then the following proposition is satisfied.
We notice here that the sum $\sum_{n \geq 1} \alpha_n$ gives an upper bound of $\e_1(I)-\e_0(I)+\ell_A(A/I)$ $($c.f. [9]$)$ and it plays a key role for a proof of our main theorem.

\begin{prop}
Suppose that $I$ is stretched.
Then the following assertions hold true.
\begin{itemize}
\item[$(1)$] $|\Lambda|=r_I-n_I$ and
\item[$(2)$] $\displaystyle  \sum_{n \geq 1}\alpha_n=\binom{r_I}{2}-\sum_{s \in \Lambda}s+|\Lambda| .$
\end{itemize}
\end{prop}

\begin{proof}
$(1)$: We have $n_I-r_I+1+\sum_{n \geq 1}\beta_n=\alpha_{r_I-1}=1$ by Lemma 2.2 $(3)$.
Thus, we get $|\Lambda|=\sum_{n \geq 1}\beta_n=r_I-n_I$ by Lemma 2.2 $(2)$ and $(3)$ as required.

\noindent
$(2)$: We set $s_0=1$, $s_{|\Lambda|+1}=r_I$, and $s_n >0$ for $1 \leq n \leq |\Lambda|$ such that $ \Lambda=\{s_n \ | \ 1 \leq n \leq |\Lambda|\}$.
For each $1 \leq m \leq |\Lambda|$, we have $\alpha_n=n_I-n+m$ for $s_m \leq n < s_{m+1}$ by Lemma 2.2 $(3)$.
Then the equalities
\begin{eqnarray*}
\sum_{n \geq 1}\alpha_n=\sum_{n=1}^{r_I-1}\alpha_n&=&\sum_{m=0}^{|\Lambda|}\sum_{n=s_m}^{s_{m+1}-1}(n_I-n+m)\\
&=& \sum_{m=0}^{|\Lambda|}\sum_{n=s_m+1}^{s_{m+1}-1}(n_I-n+m)+\sum_{m=0}^{|\Lambda|}(n_I-s_m+m)\\
&=& \binom{n_I-1}{2}+\sum_{m=0}^{|\Lambda|}(n_I+m)-\sum_{m=0}^{|\Lambda|}s_m\\
&=& \binom{n_I+|\Lambda|}{2}-(n_I-1)+r_I-\sum_{m=0}^{|\Lambda|}s_m\\ 
&=& \binom{r_I}{2}-\sum_{s \in \Lambda}s+|\Lambda|
\end{eqnarray*}
hold true.
\end{proof}

We can get the following corollary for the case where the reduction number $r_I$ attains almost minimal value $n_I+1$.

\begin{cor}
Suppose that $I$ is stretched and assume that $r_I=n_I+1$.
We set $s=\min\{n \geq 1 \ | \ Q \cap I^{n+1} \neq QI^n \}$.
Then we have $\Lambda=\{s\}$ and $\sum_{n \geq 1}\alpha_n= \binom{n_I+1}{2}-s+1$.
\end{cor}

\begin{proof}
Since $Q \cap I^{s+1} \subseteq Q \cap I^{s}=QI^{s-1}$, we have $QI^{s-1} \cap I^{s+1}=Q \cap I^{s+1} \neq QI^s$ so that $\beta_s \neq 0$.
Therefore, because $|\Lambda|=r_I-n_I=1$ by Proposition 2.3 $(1)$, we get $\Lambda=\{s\}$.
\end{proof}

\section{The structure of Sally modules}

In this paper we need the notion of Sally modules for computing to the Hilbert coefficients of ideals.
The purpose of this section is to summarize some results and techniques on the Sally modules which we need throughout this paper.
Remark that in this section $\m$-primary ideals $I$ are not necessarily stretched.  

Let $T={\rmR}(Q)=A[Qt] \subseteq A[t]$ denote the Rees algebra of $Q$.
Following Vasconcelos [21], we consider 
$$S=\rmS_Q(I)=IR/IT \cong \bigoplus_{n \geq 1}I^{n+1}/Q^nI$$ 
the Sally module of $I$ with respect to $Q$.

We give one remark about Sally modules.
See [6, 21] for further information.

\begin{rem}{$([6, 21])$}
We notice that $S$ is a finitely generated graded $T$-module and $\m^{n}S=(0)$ for all $n \gg 0$.
We have $\Ass_TS \subseteq \{\m T\}$ so that $\dim_TS=d$ if $S \neq (0)$.
\end{rem}

From now on, let us introduce some techniques, being inspired by [3, 4], which plays a crucial role throughout this paper.
See [14, Section 3] $($also [13, Section 2] for the case where $I=\m)$ for the detailed proofs.

We denote by $E(m)$, for a graded module $E$ and each $m \in \Z$, the graded module whose grading is given by $[E(m)]_n=E_{m+n}$ for all $n \in \Z$.

We have an exact sequence
$$ 0 \to K^{(-1)} \to F \overset{\varphi_{-1}}{\to} G \to R/I R+T \to 0 \ \ \ \ (\dagger_{-1})$$
of graded $T$-modules induced by tensoring the canonical exact sequence
$$ 0 \to T \overset{i}{\hookrightarrow} R \to R/T \to 0$$
of graded $T$-modules with $A/I$ where $\varphi_{-1}=A/I \otimes i$, $K^{(-1)}=\operatorname\Ker{\varphi_{-1}}$, and $F=T/IT \cong (A/I)[X_1,X_2,\cdots,X_d]$ is a polynomial ring with $d$ indeterminates over the residue class ring $A/I$.
We set $\mu_A(N)=\ell_A(N/\m N)$ denotes the minimal number of generators of an $A$-module $N$.

\begin{lem}$([14, \mathrm{Lemma} \ 3.1])$
Assume that $I \supsetneq Q$ and put $\mu=\mu_A(I/Q)$. 
Then there exists an exact sequence
$$ T(-1)^{\mu} \overset{\phi}{\to} R/T \to S(-1) \to 0 $$
as graded $T$-modules.
\end{lem}

Tensoring the exact sequence of Lemma 3.2 with $A/I$, we get an exact sequence
$$ F(-1)^{\mu} \overset{\overline{\phi}}{\to} R/I R+T \to (S/I S)(-1) \to 0$$
of graded $T$-modules, where $\overline{\phi}=A/I \otimes \phi$.

We furthermore get the following commutative diagram
\[\begin{array}{ccccc}
 (F_{0}^{\mu} \otimes F)(-1) & \to & ([R/IR+T]_1 \otimes F)(-1) & \to & 0 \\
\mapdown{\simeq} & & \mapdown{\varphi_0} & & \\
 F(-1)^{\mu} & \to & \Im \overline{\phi} & \to & 0 
\end{array}\]
of graded $T$-modules by tensoring the sequence $$F_0^{\mu} \to [R/IR+T]_1 \to 0$$ with $F$.
Then, we get the exact sequence
$$ 0 \to K^{(0)}(-1) \to ([R/IR+T]_1 \otimes F)(-1) \overset{\varphi_0}{\to} R/IR+T \to S/IS(-1) \to 0 \ \ \ (\dagger_0)$$
of graded $T$-modules where $K^{(0)}=\Ker \varphi_0$.

Notice that $\Ass_TK^{(m)} \subseteq \{\m T\}$ for all $m = -1, 0$, because $F \cong (A/I)[X_1,X_2,\cdots,X_d]$ is a polynomial ring over the residue ring $A/I$ and $[R/IR+T]_1 \otimes F$ is a maximal Cohen-Macaulay module over $F$.

We then have the following Proposition by the exact sequences $(\dagger_{-1})$ and $(\dagger_0)$.

\begin{prop}$([14, \mathrm{Lemma} \ 3.3])$
We have
\begin{eqnarray*}
\ell_A(I^n/I^{n+1})&=&\ell_A(A/[I^2+Q])\binom{n+d-1}{d-1}-\ell_A(I/[I^2+Q])\binom{n+d-2}{d-2}\\
&+& \ell_A([S/IS]_{n-1})-\ell_A(K^{(-1)}_n)-\ell_A(K^{(0)}_{n-1})
\end{eqnarray*}
for all $n \geq 0$.
\end{prop}

\vskip 2mm
We also need the notion of {\it{filtration of the Sally module}} which was introduced by M. Vaz Pinto [22] as follows.

\begin{defn}([22])
We set, for each $m \geq 1$,
$$ S^{(m)}=I^{m}t^{m-1}R/I^{m}t^{m-1}T (\cong I^{m}R/I^{m}T(-m+1)).$$
\end{defn}

We notice that $S^{(1)}=S$, and $S^{(m)}$ are finitely generated graded $T$-modules for all $m \geq 1$, since $R$ is a module-finite extension of the graded ring $T$.

The following lemma follows by the definition of the graded module $S^{(m)}$.

\begin{lem}
Let $m \geq 1$ be an integer.
Then the following assertions hold true.
\begin{itemize}
\item[$(1)$] $\m^{n} S^{(m)} = (0)$ for integers $n \gg 0$; hence ${\dim}_TS^{(m)} \leq d$.
\item[$(2)$] The homogeneous components $\{ S^{(m)}_n \}_{n \in \Z}$ of the graded $T$-module $S^{(m)}$ are given by
\[ S^{(m)}_n \cong  \left\{
\begin{array}{rl}
(0) & \quad \mbox{if $n \leq m-1 $,} \\
I^{n+1}/Q^{n-m+1}I^{m} & \quad \mbox{if $n \geq m$.}
\end{array}
\right.\]
\end{itemize}
\end{lem}

Let $L^{(m)}= T S_m^{(m)} $ be a graded $T$-submodule of $S^{(m)}$ generated by $S_m^{(m)}$ and
\begin{eqnarray*}
D^{(m)} &=& (I^{m+1}/QI^m) \otimes (A/\Ann_A(I^{m+1}/QI^m))[X_1,X_2,\cdots,X_d]\\
&\cong& (I^{m+1}/QI^m)[X_1,X_2,\cdots,X_d]
\end{eqnarray*}
for $m \geq 1$ $($c.f. [22, Section 2]$)$.

We then have the following lemma.

\begin{lem}$([22, \mathrm{Section} \ 2])$
The following assertions hold true for $m \geq 1$.
\begin{itemize}
\item[$(1)$] $S^{(m)}/L^{(m)} \cong S^{(m+1)}$ so that the sequence
$$ 0 \to L^{(m)} \to S^{(m)} \to S^{(m+1)} \to 0 $$
 is exact as graded $T$-modules.
\item[$(2)$] There is a surjective homomorphism $\theta_m: D^{(m)}(-m) \to L^{(m)}$ graded $T$-modules.
\end{itemize}
\end{lem}

For each $m \geq 1$, tensoring the exact sequence
$$ 0 \to L^{(m)} \to S^{(m)} \to S^{(m+1)} \to 0$$
and the surjective homomorphism $\theta_m:D^{(m)}(-m) \to L^{(m)}$ of graded $T$-modules with $A/I$, we get the exact sequence
$$ 0 \to K^{(m)}(-m) \to D^{(m)}/ID^{(m)}(-m) \overset{\varphi_{m}}{\to} S^{(m)}/IS^{(m)} \to S^{(m+1)}/IS^{(m+1)} \to 0 \ \ \ \ (\dagger_m) $$
of graded $F$-modules where $K^{(m)}=\Ker \varphi_{m}$.

Notice here that, for all $m \geq 1$, we have $\Ass_T K^{(m)} \subseteq \{\m T\}$ because $D^{(m)}/ID^{(m)} \cong (I^{m+1}/QI^m+I^{m+2})[X_1,X_2,\cdots,X_d]$ is a maximal Cohen-Macaulay module over $F$.

We then have the following.

\begin{prop}
The following assertions hold true:
\begin{itemize}
\item[$(1)$] We have 
\begin{eqnarray*}
\ell_A(I^n/I^{n+1}) &=& \{\ell_A(A/I^2+Q)+\sum_{m=1}^{r_I-1}\ell_A(I^{m+1}/QI^m+I^{m+2})\}\binom{n+d-1}{d-1}\\
&+& \sum_{k=1}^{r_I}(-1)^k\left\{\sum_{m=k-1}^{r_I-1}\binom{m+1}{k}\ell_A(I^{m+1}/QI^m+I^{m+2})\right\}\binom{n+d-k-1}{d-k-1}\\
&-& \sum_{m=-1}^{r_I-1}\ell_A(K^{(m)}_{n-m-1})
\end{eqnarray*}
for all $n \geq \max\{0, r_I-d+1\}$.
\item[$(2)$] $\displaystyle \e_0(I)=\ell_A(A/I^2+Q)+\sum_{m=1}^{r_I-1}\ell_A(I^{m+1}/QI^m+I^{m+2})-\sum_{m=-1}^{r_I-1}\ell_{T_{\mathcal{P}}}(K^{(m)}_{\mathcal P})$ where ${\mathcal P}=\m T $.
\end{itemize}
\end{prop}

\begin{proof}
We have, for all $n \geq 0$,
\begin{eqnarray*}
&&\ell_A([S^{(m)}/IS^{(m)}]_{n-1})\\
&=&\ell_A([D^{(m)}/ID^{(m)}]_{n-m-1})+\ell_A([S^{(m+1)}/IS^{(m+1)}]_{n-1})-\ell_A(K^{(m)}_{n-m-1})
\end{eqnarray*}
by the exact sequence $(\dagger_m)$ for $1 \leq m \leq r_I-2$ and
$$ \ell_A([S^{(r_I-1)}/IS^{(r_I-1)}]_{n-1})=\ell_A([D^{(r_I-1)}/ID^{(r_I-1)}]_{n-r_I})-\ell_A(K^{(r_I-1)}_{n-r_I})$$
by the exact sequence $(\dagger_{r_I-1})$.
We then have, by the inductive steps, 
$$\ell_A([S^{(1)}/IS^{(1)}]_{n-1}) = \sum_{m=1}^{r_I-1}\ell_A([D^{(m)}/ID^{(m)}]_{n-m-1})-\sum_{m=1}^{r_I-1}\ell_A(K^{(m)}_{n-m-1}).$$
We furthermore have 
\begin{eqnarray*}
&& \sum_{m=1}^{r_I-1}\ell_A([D^{(m)}/ID^{(m)}]_{n-m-1})\\
&=& \sum_{m=1}^{r_I-1}(-1)^k \left\{\sum_{k=0}^{m+1}\binom{m+1}{k}\ell_A(I^{m+1}/QI^m+I^{m+2})\right\} \binom{n+d-k-1}{d-k-1}\\ 
&=& \sum_{m=1}^{r_I-1}\ell_A(I^{m+1}/QI^m+I^{m+2})\binom{n+d-1}{d-1}\\
& -& \sum_{m=1}^{r_I-1}(m+1)\ell_A(I^{m+1}/QI^m+I^{m+2})\binom{n+d-2}{d-2}\\
& +& \sum_{k=2}^{r_I}(-1)^k \left\{ \sum_{m=k-1}^{r_I-1}\binom{m+1}{k}\ell_A(I^{m+1}/QI^m+I^{m+2})\right\}\binom{n+d-k-1}{d-k-1} 
\end{eqnarray*}
for all $n \geq \min\{0,r_I-d+1\}$, because 
\begin{eqnarray*}
&&\ell_A([D^{(m)}/ID^{(m)}]_{n-m-1})\\
&=&\ell_A(I^{m+1}/QI^m+I^{m+2})\binom{n-(m+1)+d-1}{d-1}\\
&=&\sum_{k=0}^{m+1}(-1)^k \binom{m+1}{k}\ell_A(I^{m+1}/QI^m+I^{m+2})\binom{n+d-k-1}{d-k-1}
\end{eqnarray*}
for all $n \geq \min\{0, m-d+2\}$ for $1 \leq m \leq r_I-1$.
Thus, we get the required equality by Proposition 3.3.

\noindent
$(2)$: We recall that $\Ass_T K^{(m)} \subseteq \{{\mathcal P}\}$ for all $-1 \leq m \leq r_I-1$ and $T/{\mathcal P} \cong (A/\m)[X_1,X_2,\cdots,X_d]$ is a polynomial ring with $d$ indeterminates over the field $A/\m$. 
Therefore, for $-1 \leq m \leq r_I-1$, thanks to [2, Corollary 4.7.8] $($see [12, Theorem 14.7] also$)$, the equality
$$ \ell_A(K^{(m)}_{n-m-1})=  \ell_{T_{\mathcal{P}}}(K_{\mathcal P}^{(m)})\binom{n+d-1}{d-1}+\mbox{$($lower terms$)$}$$
holds true for all $n \gg 0$.
Thus, we get the required equality by assertion $(1)$.
\end{proof}

Let us give one remark.

\begin{rem}
Let $n \geq 1$ be an integer.
Then $Q^{n} \cap I^{n+1}=Q^nI$ holds true if and only if $K^{(-1)}_n=(0)$.
\end{rem}

The following proposition plays a key role for a proof of the main theorem $($Theorem 1.1$)$ of this paper.

\begin{prop}
Let $m \geq 0$ and $n \geq 1$ be integers. 
Then $Q^{n}I^{m+1} \cap I^{n+m+2} \subseteq Q^{n+1}I^m+Q^{n}I^{m+2}$ holds true if $K^{(m)}_n=(0)$.
\end{prop}

\begin{proof}
Suppose that $K_n^{(0)}=(0)$ for $n \geq 1$.
Take $x \in Q^{n}I \cap I^{n+2}$ and write $x=\sum_{|\gamma|=n}y_{\gamma}a_1^{\gamma_1}a_2^{\gamma_2}\cdots a_d^{\gamma_d}$ where $y_{\gamma} \in I$, $\gamma=(\gamma_1,\gamma_2,\cdots,\gamma_d) \in \Z^d_{\geq 0}$, and $|\gamma|=\sum_{k=1}^d \gamma_k$.
Let $$f=\sum_{|\gamma|=n}\overline{y_{\gamma}} \otimes X_1^{\gamma_1}X_2^{\gamma_2}\cdots X_d^{\gamma_d} \in [(I/I^2+Q)\otimes F]_{n}$$ where $\overline{y_{\gamma}}$ denotes the image of $y_{\gamma} \in I$ in $I/I^2+Q$.
Let us look at the homomorphism
$$\varphi_0 : (I/I^2+Q) \otimes F(-1) \to  R/IR+T$$
of graded $F$-modules.
Then we have 
$$\varphi_0(f)=\sum_{|\gamma|=n}\overline{y_{\gamma}t}(\overline{a_1t})^{\gamma_1}(\overline{a_2t})^{\gamma_2} \cdots (\overline{a_dt})^{\gamma_d}=\overline{xt^{n+1}}=0$$ 
because $x \in I^{n+2}$ and $[R/IR+T]_{n+1} \cong I^{n+1}/I^{n+2}+Q^{n+1}$.
Hence $f \in K^{(0)}$ so that $f=0$.
Therefore, we get $y_{\gamma} \in I^2+Q$ for all $\gamma$ with $|\gamma|=n$ and hence $$x = \sum_{|\gamma|=n}y_{\gamma}a_1^{\gamma_1}a_2^{\gamma_2}\cdots a_d^{\gamma_d}\in Q^n[I^2+Q]=Q^{n}I^2+Q^{n+1}.$$

Suppose that $K_n^{(m)}=(0)$ for $m, n \geq 1$.
Take $x \in Q^{n}I^{m+1} \cap I^{n+m+2}$ and write $x=\sum_{|\gamma|=n}y_{\gamma}a_1^{\gamma_1}a_2^{\gamma_2}\cdots a_d^{\gamma_d}$ where $y_{\gamma} \in I^{m+1}$, $\gamma=(\gamma_1,\gamma_2,\cdots,\gamma_d) \in \Z^d_{\geq 0}$, and $|\gamma|=\sum_{k=1}^d \gamma_k$.
Let $$f=\sum_{|\gamma|=n}\overline{y_{\gamma}} \otimes X_1^{\gamma_1}X_2^{\gamma_2}\cdots X_d^{\gamma_d} \in [D^{(m)}/ID^{(m)}]_{n}$$ where $\overline{y_{\gamma}}$ denotes the image of $y_{\gamma} \in I^{m+1}$ in $I^{m+1}/QI^m+I^{m+2}$.

Let us look at the homomorphism
$$\varphi_{m} : [D^{(m)}/ID^{(m)}](-m) \to  S^{(m)}/IS^{(m)}$$
of graded $F$-modules.
Then we have $$\varphi_{m}(f)=\sum_{|\gamma|=n}\overline{y_{\gamma}t^m}(\overline{a_1t})^{\gamma_1}(\overline{a_2t})^{\gamma_2} \cdots (\overline{a_dt})^{\gamma_d}=\overline{xt^{n+m}}=0$$ because $x \in I^{n+m+2}$ and $[S^{(m)}/IS^{(m)}]_{n+m} \cong I^{n+m+1}/Q^{n+1}I^m+I^{n+m+2}$.
Hence $f \in K^{(m)}$ so that $f=0$.
Thus, we get $y_{\gamma} \in I^{m+2}+QI^{m}$ for all $\gamma$ with $|\gamma|=n$ and hence $$x = \sum_{|\gamma|=n}y_{\gamma}a_1^{\gamma_1}a_2^{\gamma_2}\cdots a_d^{\gamma_d}\in Q^n[I^{m+2}+QI^{m}]=Q^{n+1}I^m+Q^{n}I^{m+2}$$
as required.
\end{proof}

We then have the following corollary.

\begin{cor}
Let $n \geq \ell \geq 1$ be integers.
Assume that $K^{(m)}_{n-m-1}=(0)$ for all $-1 \leq m \leq n-\ell-1$.
Then we have $Q^{\ell}I^{n-\ell} \cap I^{n+1}=Q^{\ell}I^{n-\ell+1}$.
We especially have $\beta_n=0$ if $K^{(m)}_{n-m-1}=(0)$ for all $-1 \leq m \leq n-2$.
\end{cor}

\begin{proof}
We proceed by induction on $\ell$.
Suppose that $\ell=n$ then, because $K^{(-1)}_{\ell}=(0)$ by our assumption, we have $Q^{\ell}\cap I^{\ell+1}=Q^{\ell} I$ by Remark 3.8.

Assume that $\ell < n$ and that our assertion holds true for $\ell+1$.
We have $K_{n-m-1}^{(m)}=(0)$ for all $-1 \leq m \leq n-(\ell+1)-1$ so that, by the hypothesis of induction on $\ell$, we get $Q^{\ell+1}I^{n-\ell-1} \cap I^{n+1}=Q^{\ell+1}I^{n-\ell}$.
Since $K^{(n-\ell-1)}_{\ell}=(0)$ by our assumption, we get the equalities $$Q^{\ell}I^{n-\ell} \cap I^{n+1}=Q^{\ell+1}I^{n-\ell-1} \cap I^{n+1}+Q^{\ell}I^{n-\ell+1}=Q^{\ell}I^{n-\ell+1}$$
by Proposition 3.9.
\end{proof}

\section{Proof of Main Theorem}

In this section, let us introduce a proof of Theorem 1.1.

Let us begin with the following proposition.

\begin{prop}
Suppose that $I$ is stretched. 
Then the following assertions hold true:
\begin{itemize}
\item[$(1)$] We have 
\begin{eqnarray*}
\ell_A(A/I^{n+1}) &=& \{\e_0(I)+r_I-n_I\}\binom{n+d}{d}\\
&-&\{\e_0(I)-\ell_A(A/I)+\binom{r_I}{2}+r_I-n_I\}\binom{n+d-1}{d-1}\\
&+& \sum_{k=2}^{r_I}(-1)^k\binom{r_I+1}{k+1}\binom{n+d-k}{d-k}- \sum_{m=-1}^{r_I-1}\sum_{i=0}^n \ell_A(K^{(m)}_{i-m-1})
\end{eqnarray*}
for all $n \geq \max\{0, r_I-d\}$.
\item[$(2)$] $\displaystyle \sum_{m=-1}^{r_I-1}\ell_{T_{\mathcal{P}}}(K^{(m)}_{\mathcal P})=r_I-n_I=|\Lambda|$ where ${\mathcal P}=\m T $.
\end{itemize}
\end{prop}

\begin{proof}
$(1)$: We have the equalities
\begin{eqnarray*}
\ell_A(A/I^2+Q)+\sum_{m=1}^{r_I-1}\ell_A(I^{m+1}/QI^m+I^{m+2})= \e_0(I)+r_I-n_I,
\end{eqnarray*}

\begin{eqnarray*}
&&\sum_{m=0}^{r_I-1}(m+1)\ell_A(I^{m+1}/QI^m+I^{m+2})\\
&=&\ell_A(I/I^2+Q)+\sum_{m=1}^{r_I-1}(m+1)\ell_A(I^{m+1}/QI^m+I^{m+2})\\
&=&\e_0(I)-\ell_A(A/I)-n_I+1+\sum_{m=1}^{r_I-1}(m+1)\\
&=&\e_0(I)-\ell_A(A/I)+\binom{r_I}{2}+r_I-n_I,
\end{eqnarray*}
and
$$\sum_{m=k-1}^{r_I-1}\binom{m+1}{k}\ell_A(I^{m+1}/QI^m+I^{m+2})=\binom{r_I+1}{k+1}$$
for $2 \leq k \leq r_I-1$, because 
\begin{eqnarray*}
\ell_A(A/I^2+Q)&=&\ell_A(A/Q)-\ell_A(I^2+Q/Q)=\e_0(I)-n_I+1,\\ 
\ell_A(I/I^2+Q)&=&\ell_A(A/I^2+Q)-\ell_A(A/I)=\e_0(I)-\ell_A(A/I)-n_I+1,
\end{eqnarray*} 
and $\ell_A(I^{m+1}/QI^m+I^{m+2})=1$ for $1 \leq m \leq r_I-1$ by Lemma 2.1 and 2.2.
Hence, we have
\begin{eqnarray*}
\ell_A(I^i/I^{i+1}) &=& \{\e_0(I)+r_I-n_I\}\binom{i+d-1}{d-1}\\
&-&\{\e_0(I)-\ell_A(A/I)+\binom{r_I}{2}+r_I-n_I\}\binom{i+d-2}{d-2}\\
&+& \sum_{k=2}^{r_I}(-1)^k\binom{r_I+1}{k+1}\binom{i+d-k-1}{d-k-1}- \sum_{m=-1}^{r_I-1}\ell_A(K^{(m)}_{i-m-1})
\end{eqnarray*}
for all $i \geq \max\{0, r_I-d+1\}$ by Proposition 3.7 $(1)$.
Then, taking $\ell_A(A/I^{n+1})=\sum_{i=0}^n\ell_A(I^i/I^{i+1})$, we get the required equality.

\noindent
$(2)$: We have  $$\e_0(I)=\e_0(I)+r_I-n_I-\sum_{m=-1}^{r_I-1}\ell_{T_{\mathcal{P}}}(K^{(m)}_{\mathcal P})$$ by Proposition 3.7 $(2)$.
Thus, the required equality holds true.
\end{proof}

Now we get the following result of Sally and Rossi-Valla as a corollary.

\begin{cor}$([19, \mathrm{Corollary} \ 2.4], [17, \mathrm{Theorem} \ 2.6])$
Suppose that $I$ is stretched. Then the equality $r_I=n_I$ holds true if and only if the associated graded ring ${G}$ is Cohen-Macaulay.
When this is the case the following assertions also follow.
\begin{itemize}
\item[$(1)$] $\displaystyle \e_1(I)=\e_0(I)-\ell_A(A/I)+\binom{n_I}{2}$.
\item[$(2)$] $\displaystyle \e_k(I)=\binom{n_I+1}{k+1}$ for $2 \leq k \leq d$.
\item[$(3)$] The Hilbert series $ HS_I(z)$ of $I$ is given by {\small { $$ HS_I(z)=\frac{\ell_A(A/I)+\{\e_0(I)-\ell_A(A/I)-n_I+1\}z+\sum_{2 \leq k \leq n_I}z^k}{(1-z)^d}.$$}} 
\end{itemize}
\end{cor}

Let $B=T/\m T \cong (A/\m)[X_1,X_2,\cdots,X_d]$ which is a  polynomial ring with $d$ indeterminates over the field $A/\m$. 

The following proposition plays an important role for our proof of Theorem 1.1.

\begin{prop}
Suppose that $I$ is stretched and assume that $r_I=n_I+1$.
We set $s=\min\{n \geq 1 \ | \ Q \cap I^{n+1} \neq QI^n \}$.
Then the following conditions hold true:
\begin{itemize}
\item[$(1)$]$ K^{(m)} \cong B(-s+m+1)$ as graded $T$-modules and $K^{(n)}=(0)$ for all $n \neq m$ for either of $m=-1$ or $m=0$.
\item[$(2)$] $\depth G =d-1$.
\end{itemize}
\end{prop}

\begin{proof}
We have $2 \leq s \leq r_I-1$ and $\Lambda=\{s\}$ by Corollary 2.4 so that $\beta_s=1$.
Then we have $K^{(m)} \neq (0)$ for some $-1 \leq m \leq s-2$ by Corollary 3.10.
Because $\sum_{n=-1}^{r_I-1}\ell_{T_{\mathcal{P}}}(K^{(n)}_{\mathcal P})=r_I-n_I=1$ by Proposition 4.1 $(2)$, we get $\ell_{T_{\mathcal{P}}}(K^{(m)}_{\mathcal P})=1$ and $K^{(n)}=(0)$ for all $n \neq m$.

\begin{claim}
We have $-1 \leq m \leq 0$.
\end{claim}

\begin{proof}
Assume that $m \geq 1$. 
Then, $K^{(m)} \neq (0)$ implies that $K^{(m+1)} \neq (0)$ if $1 \leq m \leq r_I-2$, because $I^{n+1}/QI^n+I^{n+2}$ is generated by $x^ny$ as $A/\m$-vector space for $1 \leq n \leq r_I-1$ by Lemma 2.1 $(3)$.
Therefore, we have $m=r_I-1$, but it is impossible. 
Thus, we get $-1 \leq m \leq 0$.
\end{proof}
Then $K^{(m)}$ forms $B$-torsionfree module with $\rank_BK^{(m)}=1$ because $\Ass_{T} K^{(m)}=\{\mathcal{P}\}$ and $\ell_{T_{\mathcal{P}}}(K^{(m)}_{\mathcal{P}})=1$ so that $K^{(m)} \cong \a(h)$ as graded $B$-modules for some graded ideal $\a$ in $B$ and $h \in \mathbb{Z}$.
Because the ring $B$ is a unique factorization domain, we may choose $\a=B$ or $\mathrm{ht}_B \a \geq 2$.
Because $\a_{h+s-m-1} \cong K^{(m)}_{s-m-1} \neq (0)$ by Corollary 3.10, we have $h-m-1 \geq -s$. 
Then the equality
\begin{eqnarray*}
\sum_{i=0}^n\ell_A(K^{(m)}_{i-m-1})&=&\sum_{i=0}^n\ell_A(B_{i+h-m-1})-\sum_{i=0}^n\ell_A([B/\a]_{i+h-m-1})\\
&=& \binom{n+h-m-1+d}{d}-\sum_{i=0}^n\ell_A([B/\a]_{i+h-m-1})\\
&=& \binom{n+d}{d}+(h-m-1)\binom{n+d-1}{d-1}+\mbox{(lower terms)}
\end{eqnarray*}
holds true for all $n \gg 0$ by the canonical exact sequence
$$ 0 \to K^{(m)} \cong \a(h) \to B(h) \to (B/\a)(h) \to 0$$
of graded $B$-modules.
Hence, we have 
$$\e_1(I)=\e_0(I)-\ell_A(A/I)+\binom{n_I+1}{2}+1+(h-m-1) \geq  \e_0(I)-\ell_A(A/I)+\binom{n_I+1}{2}-s+1$$
by Proposition 4.1 $(1)$. 
Therefore, since $$\e_1(I) \leq \e_0(I)-\ell_A(A/I)+\binom{n_I+1}{2}-s+1$$ by [9, Theorem 4.7] and Corollary 2.4, the equalities $h-m-1=-s$ and 
$$\e_1(I)=\e_0(I)-\ell_A(A/I)+\binom{n_I+1}{2}-s+1$$ hold true.
Thus, we get $K^{(m)} \cong B(-s+m+1)$ because $\a_0=\a_{h-m-1+s} \cong K^{(m)}_{s-m-1} \neq (0)$, and $\depth G =d-1$ by [9, Theorem 4.7]. 
$($We can also get $\depth G=d-1$ by the exact sequences $(\dagger_m)$ for $-1 \leq m \leq r_I-1$ and the Depth Lemma$)$.
\end{proof}

\begin{proof}[Proof of Theorem 1.1]
We only need to show that assertions $(1)$, $(2)$, $(3)$, and $(4)$ hold true.
We notice that $\Lambda=\{s\}$ holds true.
Thanks to Proposition 4.3, we have
$$\sum_{m=-1}^{r_I-1}\sum_{i=0}^n\ell_A(K^{(m)}_{i-m-1})=\sum_{i=0}^n\ell_A(B_{i-s})=\binom{n-s+d}{d}=\sum_{k=0}^{s}(-1)^k\binom{s}{k}\binom{n+d-k}{d-k}$$
for all $n \geq \max\{0, s-d\}$.
Hence, by Proposition 4.1, we get
\begin{eqnarray*}
\ell_A(A/I^{n+1}) &=& \e_0(I)\binom{n+d}{d}-\{\e_0(I)-\ell_A(A/I)+\binom{n_I+1}{2}-s+1\}\binom{n+d-1}{d-1}\\
&+& \sum_{k=2}^{n_I}(-1)^k\left\{\binom{n_I+2}{k+1}-\binom{s}{k}\right\}\binom{n+d-k}{d-k}
\end{eqnarray*}
for all $n \geq \max\{0, n_I-d+1\}$.
Thus, we get the required Hilbert coefficients $\e_k(I)$ for $1 \leq k \leq d$.
We can also get the required equality of the Hilbert series $HS_I(z)$ of $I$ by the exact sequences $(\dagger_{m})$ for $-1 \leq m \leq n_I$.
\end{proof}

\section{Applications}

In this section let us introduce some applications of Theorem 1.1.
We study the structure of stretched $\m$-primary ideals with small reduction number.

\begin{rem}
Suppose that $I$ is stretched.
We notice that we have $n_I, r_I \geq 2$, and $G$ is Cohen-Macaulay if $r_I= 2$.
\end{rem}

We have the following proposition for the case where the reduction number is three.

\begin{prop}
Suppose that $I$ is stretched and assume that $r_I =3$.
Then we have $\Lambda \subseteq \{2\}$ and the following condition holds true.
\begin{itemize}
\item[$(1)$] Suppose $\Lambda=\emptyset$. Then
\begin{itemize}
\item[$(i)$] $n_I=3$, $\alpha_1=2$, $\alpha_2=1$,
\item[$(ii)$] $\e_1(I)=\e_0(I)-\ell_A(A/I)+3$, $\e_2(I)=4$ if $d \geq 2$, $\e_3(I)=1$, if $d \geq 3$, and $\e_i(I)=0$ for $4 \leq i \leq d$, and
\item[$(iii)$] $G$ is Cohen-Macaulay.
\end{itemize}
\item[$(2)$] Suppose $\Lambda=\{2\}$. Then
\begin{itemize}
\item[$(i)$] $n_I=2$, $\alpha_1=\alpha_2=1$, 
\item[$(ii)$] $\e_1(I)=\e_0(I)-\ell_A(A/I)+2$, $\e_2(I)=3$ if $d \geq 2$, $\e_3(I)=1$, if $d \geq 3$, and $\e_i(I)=0$ for $4 \leq i \leq d$, and
\item[$(iii)$] $\depth G=d-1$.
\end{itemize}
\end{itemize}
\end{prop}

\begin{proof}
Since $n_I=r_I-|\Lambda|$, the assertions $(1)$ and $(2)$ follow by Corollary 4.2 and Theorem 1.1 respectively.
\end{proof}

The following theorem determine the structure of stretched $\m$-primary ideals with reduction number four.

\begin{thm}
Suppose that $I$ is stretched and assume that $r_I =4$.
Then we have $\Lambda \subseteq \{2,3\}$ and the following conditions hold true.
\begin{itemize}
\item[$(1)$] Suppose $\Lambda=\emptyset$. Then
\begin{itemize}
\item[$(i)$] $n_I=4$, $\alpha_1=3$, $\alpha_2=2$, $\alpha_3=1$,
\item[$(ii)$]  $\e_1(I)=\e_0(I)-\ell_A(A/I)+6$, $\e_2(I)=10$ if $d \geq 2$, $\e_3(I)=5$, if $d \geq 3$, $\e_4(I)=1$ if $d \geq 4$, and $\e_i(I)=0$ for $5 \leq i \leq d$, and
\item[$(iii)$] $G$ is Cohen-Macaulay.
\end{itemize}
\item[$(2)$] Suppose $\Lambda=\{2\}$. Then
\begin{itemize}
\item[$(i)$] $n_I=3$, $\alpha_1=\alpha_2=2$, $\alpha_3=1$, 
\item[$(ii)$] $\e_1(I)=\e_0(I)-\ell_A(A/I)+5$, $\e_2(I)=9$ if $d \geq 2$, $\e_3(I)=5$, if $d \geq 3$, $\e_4(I)=1$ if $d \geq 4$, and $\e_i(I)=0$ for $5 \leq i \leq d$, and
\item[$(iii)$] $\depth G=d-1$
\end{itemize}
\item[$(3)$] Suppose $\Lambda=\{3\}$. Then
\begin{itemize}
\item[$(i)$] $n_I=3$, $\alpha_1=2$, $\alpha_2=\alpha_3=1$,
\item[$(ii)$] $\e_1(I)=\e_0(I)-\ell_A(A/I)+4$, $\e_2(I)=7$ if $d \geq 2$, $\e_3(I)=4$, if $d \geq 3$, $\e_4(I)=1$ if $d \geq 4$, and $\e_i(I)=0$ for $5 \leq i \leq d$, and
\item[$(iii)$] $\depth G=d-1$.
\end{itemize}
\item[$(4)$] Suppose $\Lambda=\{2,3\}$. Then
\begin{itemize}
\item[$(i)$] $n_I=2$, $\alpha_1=\alpha_2=\alpha_3=1$, 
\item[$(ii)$] $\e_1(I)=\e_0(I)-\ell_A(A/I)+3$, $\e_2(I)=6$ if $d \geq 2$, $\e_3(I)=4$, if $d \geq 3$, $\e_4(I)=1$ if $d \geq 4$, and $\e_i(I)=0$ for $5 \leq i \leq d$, and
\item[$(iii)$] $\depth G=d-1$.
\end{itemize}
\end{itemize}
\end{thm}

\begin{proof}
The assertions $(1)$, $(2)$, and $(3)$ follow by Corollary 4.2 and Theorem 1.1.
Suppose that $\Lambda=\{2,3\}$ then we have $\alpha_1=n_I-1=r_I-|\Lambda|-1=1$.
Thanks to [15, Theorem 2.1], [23, Theorem 3.1], and [7, Corollary 2.11], we obtain the assertion $(4)$.
\end{proof}

We can get the following corollary by Proposition 5.2 and Theorem 5.3.

\begin{cor}
Suppose that $I$ is stretched and assume that $I^5=QI^4$ $($i.e. $r_I \leq 4)$, then $G$ is almost Cohen-Macaulay.
\end{cor}

We set $\tau(I)=\ell_A([Q:I] \cap I/Q)$ and call it the Cohen-Macaulay type of the ideal $I$.

The following lemma is given by Rossi and Valla [17].

\begin{lem}$([17, \mathrm{Theorem} \ 2.7])$
Suppose that $I$ is a stretched and assume that $\tau(I)<\ell_A(I/I^2)-(d+1)\ell_A(A/I)+1$.
Then, $\alpha_2=\alpha_1-1$ and $\beta_2=0$ hold true.
\end{lem}

The following corollaries hold true by Proposition 5.2, Theorem 5.3, and Lemma 5.5.

\begin{cor}
Suppose that $I$ is stretched and assume that $\tau(I) < \ell_A(I/I^2)-(d+1)\ell_A(A/I)+1$.
If $r_{I} =3$, then the following conditions hold true.
\begin{itemize}
\item[$(1)$] $n_I=3$, $\alpha_1=2$, $\alpha_2=1$,
\item[$(2)$] $\e_1(I)=\e_0(I)-\ell_A(A/I)+3$, $\e_2(I)=4$ if $d \geq 2$, $\e_3(I)=1$, if $d \geq 3$, and $\e_i(I)=0$ for $4 \leq i \leq d$.
\item[$(3)$] $G$ is Cohen-Macaulay.
\end{itemize}
\end{cor}

\begin{cor}
Suppose that $I$ is stretched and assume that $\tau(I) < \ell_A(I/I^2)-(d+1)\ell_A(A/I)+1$.
If $r_{I} =4$, then we have $\Lambda \subseteq \{3\}$ and the following condition holds true.
\begin{itemize}
\item[$(1)$] Suppose $\Lambda=\emptyset$. Then
\begin{itemize}
\item[$(i)$] $n_I=4$, $\alpha_1=3$, $\alpha_2=2$, and $\alpha_3=1$,
\item[$(ii)$]  $\e_1(I)=\e_0(I)-\ell_A(A/I)+6$, $\e_2(I)=10$ if $d \geq 2$, $\e_3(I)=5$, if $d \geq 3$, $\e_4(I)=1$ if $d \geq 4$, and $\e_i(I)=0$ for $5 \leq i \leq d$.
\item[$(iii)$] $G$ is Cohen-Macaulay.
\end{itemize}
\item[$(2)$] Suppose $\Lambda=\{3\}$. Then
\begin{itemize}
\item[$(i)$] $n_I=3$, $\alpha_1=2$, $\alpha_2=\alpha_3=1$,
\item[$(ii)$]  $\e_1(I)=\e_0(I)-\ell_A(A/I)+4$, $\e_2(I)=7$ if $d \geq 2$, $\e_3(I)=4$, if $d \geq 3$, $\e_4(I)=1$ if $d \geq 4$, and $\e_i(I)=0$ for $5 \leq i \leq d$.
\item[$(iii)$] $\depth G=d-1$.
\end{itemize}
\end{itemize}
\end{cor}

Let $\tau(A)=\tau(\m)=\ell_A([Q:\m]/Q)$ denote the Cohen-Macaulay type of  $A$.
Let $k[[u]]$ denote the formal power series ring with one indeterminate $u$ over a field $k$.

The following example satisfies the conditions of Corollary 5.6.

\begin{ex}
Let $A=k[[u^{7},u^{15},u^{18},u^{26},u^{27}]]$ and $Q=(u^7)$.
Then $Q$ is a reduction of the maximal ideal $\m$ of $A$, $\m^2=Q\m+(u^{30})$, $\tau(A)=2$, $\mu(\m)=5$, and $r_{\m}=3$.
Therefore, $\ell_A(A/\m^{n+1})=7(n+1)-9$ for all $n \geq 2$ and $\rmG(\m)$ is Cohen-Macaulay.
\end{ex}

We can also construct examples which satisfy the condition of $(1)$ and $(2)$ of Corollary 5.7 respectively.

\begin{ex}
The following assertions hold true.
\begin{itemize}
\item[$(1)$] Let $A=k[[u^{8},u^{17},u^{29},u^{38},u^{39}]]$ and $Q=(u^8)$. 
Then $Q$ is a reduction of the maximal ideal $\m$ of $A$, $\m^2=Q\m+(t^{34})$, $\tau(A)=2$, $\mu(\m)=5$, and $r_{\m}=n_{\m}=4$.
Therefore, $\ell_A(A/\m^{n+1})=8(n+1)-13$ for all $n \geq 3$ and $\rmG(\m)$ is Cohen-Macaulay.

\item[$(2)$] Let $A=k[[u^{8},u^{17},u^{21},u^{30},u^{39},u^{52}]]$ and $Q=(u^8)$.
Then $Q$ is a reduction of the maximal ideal $\m$ of $A$, $\m^2=Q\m+(u^{34})$, $\tau(A)=3$, $\mu(\m)=6$, $r_{\m}=4$, and $n_{\m}=3$.
Therefore, $\ell_A(A/\m^{n+1})=8(n+1)-11$ for all $n \geq 3$ and $\rmG(\m)$ is not Cohen-Macaulay.
\end{itemize}
\end{ex}

Let us now introduce the following theorem.

\begin{thm}
Suppose that $I$ is stretched and $\tau(I) < \ell_A(I/I^2)-(d+1)\ell_A(A/I)+1$.
If $r_{I} =5$, then we have $\Lambda \subseteq \{3,4\}$ and the following conditions hold true.
\begin{itemize}
\item[$(1)$] Suppose $\Lambda=\emptyset$. Then
\begin{itemize}
\item[$(i)$] $n_I=5$ and $\alpha_1=4$, $\alpha_2=3$, $\alpha_3=2$, and $\alpha_4=1$,
\item[$(ii)$] $\e_1(\m)=\e_0(I)-\ell_A(A/I)+10$, $\e_2(I)=20$ if $d \geq 2$, $\e_3(I)=15$, if $d \geq 3$, $\e_4(I)=6$ if $d \geq 4$, $\e_5(I)=1$ if $d \geq 5$, and $\e_i(I)=0$ for $6 \leq i \leq d$.
\item[$(iii)$] $G$ is Cohen-Macaulay.
\end{itemize}
\item[$(2)$] Suppose $\Lambda=\{3\}$. Then 
\begin{itemize}
\item[$(i)$] $n_I=4$ and $\alpha_1=3$, $\alpha_2=\alpha_3=2$, $\alpha_4=1$, 
\item[$(ii)$] $\e_1(I)=\e_0(I)-\ell_A(A/I)+8$, $\e_2(I)=17$ if $d \geq 2$, $\e_3(I)=14$, if $d \geq 3$, $\e_4(I)=6$ if $d \geq 4$, $\e_5(I)=1$ if $d \geq 5$, and $\e_i(I)=0$ for $6 \leq i \leq d$.
\item[$(iii)$] $\depth G=d-1$.
\end{itemize}
\item[$(3)$] Suppose $\Lambda=\{4\}$. Then
\begin{itemize}
\item[$(i)$] $n_I=4$ and $\alpha_1=3$, $\alpha_2=2$, $\alpha_3=\alpha_4=1$,
\item[$(ii)$] $\e_1(I)=\e_0(I)-\ell_A(A/I)+7$, $\e_2(I)=14$ if $d \geq 2$, $\e_3(I)=11$, if $d \geq 3$, $\e_4(I)=5$ if $d \geq 4$, $\e_5(I)=1$ if $d \geq 5$, and $\e_i(I)=0$ for $6 \leq i \leq d$,
\item[$(iii)$] $\depth G=d-1$.
\end{itemize}

\item[$(4)$] Suppose $\Lambda=\{3,4\}$. Then
\begin{itemize}
\item[$(i)$] $n_I=3$ and $\alpha_1=2$, $\alpha_2=\alpha_3=\alpha_4=1$, 
\item[$(ii)$] $\e_1(I)=\e_0(I)-\ell_A(A/I)+5$, $\e_2(I)=11$ if $d \geq 2$, $\e_3(I)=10$, if $d \geq 3$, $\e_4(I)=5$ if $d \geq 4$, $\e_5(I)=1$ if $d \geq 5$, and $\e_i(I)=0$ for $6 \leq i \leq d$.
\item[$(iii)$] $\depth G=d-1$.
\end{itemize}
\end{itemize}
\end{thm}

\begin{proof}
The assertions $(1)$, $(2)$, and $(3)$ follow by Corollary 4.2 and Theorem 1.1.
Suppose that $\Lambda=\{3,4\}$ then we have $\alpha_1=n_I-1=r_I-|\Lambda|-1=2$, and $\alpha_2=\alpha_1-1=1$ by Lemma 5.5.
We notice that $Q \cap I^3=QI \cap I^3=QI^2$ hold true because $\beta_1=\beta_2=0$.
Then the assertion $(4)$ is obtained by [8, Theorem 2.4] $($see [5, Theorem 3.10] also$)$.
\end{proof}

\begin{cor}
Suppose that $I$ is stretched and assume that $\tau(I) < \ell_A(I/I^2)-(d+1)\ell_A(A/I)+1$.
Then $G$ is almost Cohen-Macaulay, if $I^6=QI^5$ $($i.e. $r_I \leq 5)$ holds true.
\end{cor}

In the end of this section, let us introduce some examples of stretched Cohen-Macaulay local ring $(A, \m)$ which satisfies $\tau(A) < \mu(\m)-d$ and $r_{\m}=5$ as follows.

\begin{ex}
The following assertions hold true.
\begin{itemize}
\item[$(1)$] Let $A=k[[u^{9},u^{19},u^{42},u^{52},u^{53}]]$ and $Q=(u^9)$.
Then $Q$ is a reduction of the maximal ideal $\m$ of $A$, $\m^2=Q\m+(u^{38})$, $\tau(A)=2$, and $\mu(\m)=r_{\m}=n_{\m}=5$. 
Therefore, $\ell_A(A/\m^{n+1})=9(n+1)-18$ for all $n \geq 4$ and $\rmG(\m)$ is Cohen-Macaulay.

\item[$(2)$] Let $A=k[[u^{9},u^{19},u^{33},u^{43},u^{53},u^{68}]]$ and $Q=(u^9)$.
Then $Q$ is a reduction of the maximal ideal $\m$ of $A$, $\m^2=Q\m+(u^{38})$, $\tau(A)=3$, $\mu(\m)=6$, $r_{\m}=5$, $n_{\m}=4$, and $\Lambda_{\m}=\{4\}$.
Therefore, $\ell_A(A/\m^{n+1})=9(n+1)-16$ for all $n \geq 4$ and $\rmG(\m)$ is not Cohen-Macaulay.

\item[$(3)$] Let $A=k[[u^{9},u^{19},u^{33},u^{43},u^{53},u^{77}]]$ and $Q=(u^9)$.
Then $Q$ is a reduction of the maximal ideal $\m$ of $A$, $\m^2=Q\m+(u^{38})$, $\tau(A)=3$, $\mu(\m)=6$, $r_{\m}=5$, $n_{\m}=4$, and $\Lambda_{\m}=\{3\}$.
Therefore, $\ell_A(A/\m^{n+1})=9(n+1)-15$ for all $n \geq 4$ and $\rmG(\m)$ is not Cohen-Macaulay.

\item[$(4)$] Let $A=k[[u^{10},u^{21},u^{26},u^{37},u^{48},u^{59},u^{64},u^{75}]]$ and $Q=(u^{10})$.
Then $Q$ is a reduction of the maximal ideal $\m$ of $A$, $\m^2=Q\m+(u^{42})$, $\tau(A)=5$, $\mu(\m)=8$, $r_{\m}=5$, $n_{\m}=3$, and $\Lambda_{\m}=\{3, 4\}$.
Therefore, $\ell_A(A/\m^{n+1})=9(n+1)-13$ for all $n \geq 4$ and $\rmG(\m)$ is not Cohen-Macaulay.
\end{itemize}
\end{ex}

\section{Example}

Stretched ideals could have arbitrarily high reduction number against the index of nilpotency. 
In this section we analyze a few examples of one-dimensional stretched Cohen-Macaulay local rings $(A, \m)$.

Let $b$, $e$ and $\ell$ are integers with $b \geq 2$ and $2 \leq \ell \leq e-1$.
We set $\{b_n \in \mathbb{Z} \ | \  n=1, \ell+1, \ell+2, \cdots,e-1 \}$ be integers with $b_1=b$, $\lceil \frac{b}{2} \rceil n +1 \leq b_n$ for all $\ell+1 \leq n \leq e-1$, $b_{\ell+1} \leq b\ell+b-1$, and $b_{n+1} \leq b_{n}+\lceil \frac{b}{2} \rceil$ for all $\ell+1 \leq n \leq e-2$ where $\lceil q \rceil =\min\{n \in \Z \ | \ n \geq q \}$ for a quotient number $q \in \mathbb{Q}$.
We notice here that the inequalities $\lceil \frac{b}{2} \rceil n+1 \leq b_n \leq (b-1)n+\ell$ hold true for $\ell+1 \leq n \leq e-1$.
We set $r=\max\{ n < e \ | \ b_n > bn-n+1 \} \cup \{\ell\}$.

Let $$H=\langle e, b_1e+1, b_ne+n \ | \ \ell+1 \leq n \leq e-1 \rangle$$ be the numerical semi-group generated by $e$, $b_1e+1$, and $b_{n}e+n$ for $\ell+1 \leq n \leq e-1$.
We set $$C=k[H]=k[ u^e, u^{be+1},u^{  b_{\ell+1}e+(\ell+1) },\cdots,u^{b_{e-1}e+e-1} ] \subseteq k[u]$$ be a numerical semi-group ring of $H$, where $k[u]$ denotes the polynomial ring with one indeterminate $u$ over a field $k$.
Let $a=a_0=u^e$, $a_1=u^{be+1}$, and $a_n=u^{b_ne+n}$ for $\ell+1 \leq n \leq e-1$ and put $\q=aC$ and $\n=(a,a_1,a_{\ell+1},a_{\ell+2},\cdots,a_{e-1})C$ is the graded maximal ideal of $C$.

We then have the following lemmas.

\begin{lem}
Let $n_1,n_2 \in \{1, \ell+1,\ell+2,\cdots,e-1\}$ be integers.
Then the following assertions hold true:
\begin{itemize}
\item[$(1)$] $a_{n_1}a_{n_2}=a^{b_{n_1}+b_{n_2}-b_{n_1+n_2}}a_{n_1+n_2}$ holds true if $n_1+n_2 <e$. 
\item[$(2)$] $a_{n_1}a_{n_2}=a^{b_{n_1}+b_{n_2}+1-bq_{n_1+n_2}}a_{1}^{q_{n_1+n_2}}$ holds true if $n_1+n_2 \geq e$ where $q_{n_1+n_2}$ be an integer with $0 \leq q_{n_1+n_2} \leq e-2$ and $n_1+n_2=e+q_{n_1+n_2}$.
\end{itemize}
\end{lem}

\begin{proof}
Suppose that $n_1+n_2 \leq e-1$.
Then we have 
\begin{eqnarray*}
a_{n_1} a_{n_2}&=&u^{b_{n_1}e+n_1}u^{b_{n_2}e+n_2}=u^{(b_{n_1}e+n_1)+(b_{n_2}e+n_2)-b_{n_1+n_2}e+b_{n_1+n_2}e}\\
&=&u^{(b_{n_1}+b_{n_2}-b_{n_1+n_2})e+b_{n_1+n_2}e+(n_1+n_2)}=a^{b_{n_1}+b_{n_2}-b_{n_1+n_2}}a_{n_1+n_2}.
\end{eqnarray*}

Suppose that $n_1+n_2 \geq e$ then we may write $n_1+n_2=e+q_{n_1+n_2}$ for some $0 \leq q_{n_1+n_2} \leq e-2$.
Then the equalities 
\begin{eqnarray*}
a_{n_1}a_{n_2}&=&u^{b_{n_1}e+n_1}u^{b_{n_2}e+n_2}
=u^{(b_{n_1}+b_{n_2})e+e+q_{n_1+n_2}-bq_{n_1+n_2}e+bq_{n_1+n_2}e}\\
&=&u^{(b_{n_1}+b_{n_2}+1-bq_{n_1+n_2})e+(be+1)q_{n_1+n_2}}
= a^{b_{n_1}+b_{n_2}+1-bq_{n_1+n_2}}a_1^{q_{n_1+n_2}}
\end{eqnarray*}
hold true.
\end{proof}

\begin{lem}
The following assertions hold true.
\begin{itemize}
\item[$(1)$] $a_1^{\ell} \notin \q$ and $a_1^e \in \q^{be+1}$,
\item[$(2)$] $\n^{n}=\q \n^{n-1}+a_1^{n}C$ holds true for all $n \geq 2$.
\end{itemize}
Therefore, we have $\n^{e}=\q\n^{e-1}$ so that $\q$ is a reduction of $\n$, and  $\ell_C(\n^2+\q/\n^3+\q)=1$.
\end{lem}

\begin{proof}
$(1)$: Assume that $a_1^{\ell}=u^{(be+1)\ell} \in \q$ then we have $u^{(be+1)\ell-e}=u^{(b\ell-1)e+\ell} \in C$. 
However, it is impossible.
It is easy to see that $a_1^e=u^{(be+1)e}=a^{be+1} \in \q^{be+1}$ holds true.

\noindent
$(2)$: We have only to show that the equality 
$$\n^2=\q\n+a_1^2C+a_1(a_n \ | \ \ell \leq n \leq e-1)C+(a_{n_1}a_{n_2} \ | \ \ell \leq n_1, n_2 \leq e-1)C=\q\n+a_1^2C$$
holds true.
Because the inequalities $$b_{n}-b_{n+1}+b > b_n-b_{n+1}+\left \lceil \frac{b}{2} \right \rceil \geq 0$$ holds true by our assumption, we get 
$a_{1}a_{n}=a^{b_n+b_1-b_{n+1}}a_{n+1} \in \q\n$ for $\ell+1 \leq n  \leq e-2$ by Lemma 6.1 $(1)$.
It is easy to check that $a_1a_{e-1}=a^{b_{e-1}+b_1+1} \in \q^2$ holds true.

Let $\ell+1 \leq n_1,n_2 \leq e-1$ be integers and assume that $n_1+n_2 <e$.
Then since $b_{n_1} \geq b_{n_1+n_2}-\lceil \frac{b}{2} \rceil n_2$ and $b_{n_2} \geq b_{n_1+n_2}-\lceil \frac{b}{2} \rceil n_1$, we have $$b_{n_1}+b_{n_2}-b_{n_1+n_2} \geq b_{n_1+n_2}-\left \lceil \frac{b}{2} \right \rceil (n_1+n_2) > 0$$ by our assumption.
Hence, we get  $a_{n_1}a_{n_2}=a^{b_{n_1}+b_{n_2}-b_{n_1+n_2}}a_{n_1+n_2} \in \q\n$ by Lemma 6.1 $(1)$.

Assume that $n_1 +n_2>e$ and write $n_1+n_2=e+q_{n_1+n_2}$ for $0 \leq q_{n_1+n_2} \leq e-2$.
We then have 
\begin{eqnarray*}
b_{n_1}+b_{n_2} &\geq& \left\lceil \frac{b}{2} \right\rceil(n_1+n_2)+2 \geq \frac{b}{2} (e+q_{n_1+n_2})+2 \\
&\geq& \frac{b}{2}(2q_{n_1+n_2}+2)+2=bq_{n_1+n_2}+b+2
\end{eqnarray*}
 by our assumption.
Hence $b_{n_1}+b_{n_2}+1-bq_{n_1+n_2}>0$ holds true.
Therefore, we get  $a_{n_1}a_{n_2}=a^{b_{n_1}+b_{n_2}+1-bq_{n_1+n_2}}a_{1}^{q_{n_1+n_2}} \in \q\n$ by Lemma 6.1 $(2)$.

Thus, the required equality $\n^2=\q\n+a_1^2C$ is satisfied.
\end{proof}

We furthermore have the following proposition which is related to the reduction number of the graded maximal ideal $\n$ of $C$.

\begin{prop}
Let $m > n>0$ be integers with $\ell+1 \leq m \leq e-1$. 
Then the inequality $b_{m} \leq bm-n$ holds true if and only if $a_1^{m} \in \q\n^{n}$.
\end{prop} 

\begin{proof}
Suppose that the inequality $b_m \leq bm-n$ holds true.
Then we have $$a_1^m=u^{(be+1)m}=u^{(bm-b_m)e+b_me+m}=a^{bm-b_m}a_m \in \q \n^n.$$

On the other hand, suppose that $a_1^m \in \q\n^n$ and write $a_1^m=\tau a \prod_{k=1}^{n} a_{i_k}$ for some $\tau \in C$ and $i_k \in \{ 0,1,\ell+1, \ell+2,\cdots,e-1 \}$ for $1 \leq k \leq n$.

We proceed by induction on $m$.
Assume that $m=\ell+1$. 
Then since $a_1^{\ell} \notin \q$ by Lemma 6.2 $(1)$, we have $i_k \neq 1$ for all $1 \leq k \leq n$.
Therefore, since $a_{n_1}a_{n_2} \in \q\n$ for $\ell+1 \leq n_1,n_2 \leq e-1$ by the proof of Lemma 6.2, 
we may choose $a_{i_k}=a$ for all $1 \leq k \leq n-1$ so that $a_1^{m}=u^{bme+m} \in \q^n=u^{ne}C$.
Therefore, we get $bm-n \geq b_m$ because $u^{(bm-n)e+m} \in C$.

Assume that $m > \ell+1$ and that our assertion holds true for $m-1$.
If $a_{i_k}=a_1$ for some $1 \leq k \leq m$ then $a_1^{m-1} \in \q\n^{n-1}$.
Thus, by the hypothesis of induction on $m$, we get $$b_m \leq b_{m-1}+\left\lceil \frac{b}{2} \right\rceil \leq (b(m-1)-n+1)+ \left \lceil \frac{b}{2} \right \rceil \leq bm-n$$ as required.
We may take $i_k \in \{0,\ell+1,\ell+2,\cdots,e-1 \}$ so that we get the required inequality $b_m \leq bm-n$ by the same argument as above.
\end{proof}

The following corollary follows by Proposition 6.3.

\begin{cor}
Let $\ell+1 \leq n \leq e-1$ be an integer.
Then the following assertions hold true.
\begin{itemize}
\item[$(1)$] $a_1^n \in \q\n^{bn-b_n} \backslash \q\n^{bn-b_n+1}$ holds true.
\item[$(2)$] $\n^n=\q\n^{n-1}$ if and only if $b_n \leq bn-n+1$.
\end{itemize}
\end{cor}

We set $W_n:= \q\n^{n-1} \cap \n^{n+1}/\q\n^n$ for $n \geq 1$.
Notice that $b{n_1}-b_{n_1}+1 < bn_2-b_{n_2} +1$ for all $\ell+1 \leq n_1 < n_2 \leq e-1$, and recall that $r=\max\{ n < e \ | \ b_n > bn-n+1 \} \cup \{\ell\}$.

Then we have the following lemma.

\begin{lem}
$W_{bn-b_n+1} \neq (0)$ for all $\ell+1 \leq n \leq r$.
\end{lem}

\begin{proof}
Let $\ell+1 \leq n \leq r$ and notice that the inequality $bn-b_n+2 \leq n$ holds true.
Let us show the inclusion $$ \q\n^{bn-b_n}\cap \n^{m} \subseteq  \q\n^{bn-b_n+1}+a_1^{n}C$$
holds true for all $bn-b_n+2 \leq m \leq n$ by induction on $m$.

We may assume that $m < n$ and that our assertion holds true for $m+1$.
We have
$$\q\n^{bn-b_n} \cap \n^{m} \subseteq \q\n^{bn-b_n+1}+\q\n^{bn-b_n} \cap a_1^{m}C.$$
Let $\tau a_1^{m} \in \q\n^{bn-b_n} \cap a_1^{m}C$ for some $\tau \in C$.
If $a_1^m \in \q\n^{bn-b_n}$ then we have $a_1^m \in \q\n^{bm-b_m+1}$ since $bm-b_m+1 \leq bn-b_n$.
However, it is impossible by Corollary 6.4 $(1)$ and hence we get $\tau \in \n$.
Then, by the hypothesis of induction on $m$, the required inclusion $$\tau a_1^{m} \in \q\n^{bn-b_n} \cap \n^{m+1} \subseteq \q\n^{bn-b_n+1}+a_1^nC$$ holds true.
Thus, because $$\q \n^{bn-b_n} \cap \n^{bn-b_n+2}=\q \n^{bn-b_n+1}+a_1^{n}C$$ and $a_1^{n} \in \q\n^{bn-b_n} \backslash \q\n^{bn-b_n+1}$ by Lemma 6.4 $(1)$, we get $W_{bn-b_n+1} \neq (0)$ as required.
\end{proof}

Let $A=C_{\n}$ be a localization of $C$ at $\n$ and $\m=\n A$ the maximal ideal of $A$.
We set $Q=\q A$ and then $Q$ is a parameter ideal in $A$ which forms reduction of $\m$ by Lemma 6.2.
We then have the following theorem.

\begin{thm}
The following assertions hold true.
\begin{itemize}
\item[$(1)$] $A$ is a stretched Cohen-Macaulay local ring with $\dim A=1$, $\tau(A)=e-\ell$, and $\mu_A(\m)=e-\ell+1$.
\item[$(2)$] $n_{\m}=n_Q(\m)=\ell$ and $r_{\m}=r_Q(\m)=r$.
\item[$(3)$] $\Lambda_{\m}=\Lambda_Q(\m)=\{bn-b_n+1 \ | \ \ell+1 \leq n \leq r \}.$
\item[$(4)$] $\ell_A(A/\m^{n+1})=e(n+1)-(e-1+\binom{\ell}{2}+\sum_{n=\ell+1}^{r}(b_n-bn+n-1))$ for all $n \geq r-1$.
\item[$(5)$] The following two conditions are equivalent to each other:
\begin{itemize}
\item[$(i)$] ${\rmG}(\m)$ is Cohen-Macaulay and
\item[$(ii)$] $b_{\ell+1} \leq b\ell+b-\ell$ that is the equality $r=\ell$ holds true.
\end{itemize}
\end{itemize}
\end{thm}

\begin{proof}
$(1)$: Because $[\q:\n]/\q$ is spanned by $a_1^{\ell-1}$, $a_{\ell+1}, a_{\ell+2},\cdots,a_{e-1}$ as an $C/\n$-vector space, we get $\tau(A)=e-\ell$.

\noindent
$(2)$: It follows by Lemma 6.2 and Corollary 6.4.

\noindent
$(3)$: Because $W_{bn-b_n+1} \neq (0)$ for $\ell+1 \leq n \leq r$ by Lemma 6.5 and $|\Lambda_{\m}|=r_{\m}-n_{\m}=r-\ell$ by Proposition 2.3 $(1)$, we have $\Lambda_{\m}=\{bn-b_n+1 \ | \ \ell+1 \leq n \leq r \}$.

\noindent
$(4)$: We have 
\begin{eqnarray*}
 \sum_{n \geq 1}\ell_A(\m^{n+1}/Q\m^n)&=&\binom{r}{2}-\sum_{n=\ell+1}^{r}(bn-b_n+1)+r-\ell\\
&=&\binom{r}{2}-\sum_{n=\ell+1}^{r}n-\sum_{n=\ell+1}^{r}(bn-n-b_n+1)+r-\ell\\
&=&\binom{\ell}{2}+\sum_{n=\ell+1}^{r}(b_n-bn+n-1)
\end{eqnarray*}
because $$\sum_{n \geq 1} \ell_A(\m^{n+1}/Q\m^n)=\binom{r_{\m}}{2}-\sum_{n \in \Lambda_{\m}}n+|\Lambda_{\m}|$$ by Proposition 2.3 $(2)$.
Thus, since $\ell_A(\m/Q)=e-1$ and by [9, Theorem 4.7], we get the equality of the first Hilbert coefficient 
$$\e_1(\m)=\sum_{n \geq 0}\ell_A(\m^{n+1}/Q\m^n) =e-1+\binom{\ell}{2}+\sum_{n=\ell+1}^{r}(b_n-bn+n-1)$$
of $\m$.
\end{proof}

Thanks to Theorem 6.6, we can construct the following examples of stretched maximal ideal which satisfy conditions of Theorem 1.1 as follows.

\begin{cor}
Let $b, e, \ell$, and $s$ are integers with $b \geq 2$ and $2 \leq s \leq \ell \leq e-3$.
We set $b_{\ell+1}=b(\ell+1)+1-s$ and $b_{\ell+2}=(b-1)(\ell+2)+1$.
Then we have the following.
\begin{itemize}
\item[$(1)$] $n_{\m}=n_{Q}(\m)=\ell$ and $r_{\m}=r_Q(\m)=\ell+1$.
\item[$(2)$] $\Lambda_{\m}=\Lambda_Q(\m)=\{s\}$.
\item[$(3)$] $\ell_A(A/\m^{n+1})=e(n+1)-(e-1+\binom{\ell+1}{2}-s+1)$ for all $n \geq \ell$.
\item[$(4)$] $\rmG(\m)$ is not Cohen-Macaulay.
\end{itemize}
\end{cor}

In the end of this paper, let us introduce some concrete examples of stretched local rings which satisfy conditions of Proposition 5.2 or Theorem 5.3 in Section 5.

\begin{ex}
Let $e=6$ and $b=2$.
Then the following assertions hold true.
\begin{itemize}
\item[$(1)$] Let $\ell=3$, $b_4=5$, and $b_5=6$. 
Then $H=\langle 6,13,34,41 \rangle$ and $r=3$.
Therefore, $r_{\m}=n_{\m}=3$, $\Lambda_{\m}=\emptyset$,
$\ell_A(A/\m^{n+1})=6(n+1)-8$ for all $n \geq 2$, and $\rmG(\m)$ is Cohen-Macaulay.
\item[$(2)$] Let $\ell=2$, $b_3=b_4=5$, and $b_5=6$.
Then $H=\langle 6,13,33, 34,41 \rangle$ and $r=3$.
Therefore, $r_{\m}=3$, $n_{\m}=2$, $\Lambda_{\m}=\{2\}$, $\ell_A(A/\m^{n+1})=6(n+1)-7$ for all $n \geq 2$, and ${\rmG}(\m)$ is not Cohen-Macaulay.
\item[$(3)$] Let $\ell=4$ and $b_5=6$.
Then $H=\langle 6,13,41 \rangle$ and $r=4$.
Therefore, $r_{\m}=n_{\m}=4$, $\Lambda_{\m}=\emptyset$, $\ell_A(A/\m^{n+1})=6(n+1)-11$ for all $n \geq 3$, and $\rmG(\m)$ is Cohen-Macaulay.
\item[$(4)$] Let $\ell=3$, $b_4=7$, and $b_5=6$.
Then $H=\langle 6,13,46, 41 \rangle$ and $r=4$.
Therefore, $r_{\m}=4$, $n_{\m}=3$, $\Lambda_{\m}=\{2\}$, $\ell_A(A/\m^{n+1})=6(n+1)-10$ for all $n \geq 3$, and $\rmG(\m)$ is not Cohen-Macaulay.
\item[$(5)$] Let $\ell=3$ and $b_4=b_5=6$.
Then $H=\langle 6,13,40,41 \rangle$ and $r=4$.
Therefore, $r_{\m}=4$, $n_{\m}=3$, $\Lambda_{\m}=\{3\}$, $\ell_A(A/\m^{n+1})=6(n+1)-9$ for all $n \geq 3$, and $\rmG(\m)$ is not Cohen-Macaulay.
\item[$(6)$] Let $\ell=2$, $b_3=5$, and $b_4=b_5=6$.
Then $H=\langle 6,13,33,40,41 \rangle$ and $r=4$.
Therefore, $r_{\m}=4$, $n_{\m}=2$, $\Lambda_{\m}=\{2,3\}$, $\ell_A(A/\m^{n+1})=6(n+1)-8$ for all $n \geq 3$, and $\rmG(\m)$ is not Cohen-Macaulay.
\end{itemize}
\end{ex}

%%%%%%%%%%%%%%%%%%%%%%%%%%%%%%%%%%%%%%%%%%%%%%%%%%%%%%%%%%%%%%%%%%%%%%%%%%%%%%%%%%%%%%%%%%%%%%%%%%%%%%%%%%%%%%%%%%%%%%%%%%%%%%%%%%%%%%%%%%%%%%%%%%%%%%%%%%%%%%%%%%%%%%%%%%%%%%%%%%%%%%%%%%%%%%%%%%%%%%%%%%%%%%%%%%%%%%%%%%%%%%%%%%%%%%%%%%%%%%%%%%%%%%%%%%%%%%%%%%%%%%%%%%%%%%%%%%%%%%%%%%%%%%%%%%%%%%%%%%%%%%%%%%%%%%%%%%%%%%%%%%%%%%%%%%%%%%%%%%%%%%%%%%%%%%%%%%%%%%%%%%%%%%%%%%%%%%%%%%%%%%%%%%%%%%%%%%%%%%%%

\end{document}